\documentclass[oneside, a4paper]{amsart}

\usepackage{amssymb,amsfonts,amsmath,tikz,tikz-cd,graphicx}
\usepackage[all,arc]{xy}
\usepackage{enumerate}
\usepackage{mathrsfs}
\usepackage[colorlinks]{hyperref}
\usepackage{mathtools}
\usetikzlibrary{decorations.pathmorphing,snakes}
\usepackage[alphabetic]{amsrefs}

\usetikzlibrary{knots}

\newtheorem{thm}{Theorem}[section]
\newtheorem{cor}[thm]{Corollary}
\newtheorem{prop}[thm]{Proposition}
\newtheorem{lem}[thm]{Lemma}

\newtheorem{quest}[thm]{Question}

\theoremstyle{definition}
\newtheorem{defn}[thm]{Definition}

\newtheorem{exmp}[thm]{Example}

\theoremstyle{remark}

\newcommand{\ZZ}{\mathbb{Z}/2\mathbb{Z}}

\newcommand{\RR}{\mathbb{R}}

\newcommand{\bfv}{\mathbf{v}}
\newcommand{\bfw}{\mathbf{w}}

\DeclareMathOperator{\rankk}{rank}

\DeclareMathOperator{\II}{I}

\DeclareMathOperator{\AHI}{AHI}
\DeclareMathOperator{\AKh}{AKh}

\DeclareMathOperator{\idd}{id}
\DeclareMathOperator{\Kh}{Kh}
\DeclareMathOperator{\Khr}{Khr}
\DeclareMathOperator{\CKhr}{CKhr}

\makeatletter
\let\c@equation\c@thm
\makeatother
\numberwithin{equation}{section}

\title{Annular Khovanov Homology and Augmented Links}

\author{Hongjian Yang}
\email{yanghongjian@pku.edu.cn}
\address{School of Mathematical Sciences, Peking University, Beijing 100871, China}

\begin{document}

\begin{abstract}

 Given an annular link $L$, there is a corresponding augmented link $\widetilde{L}$ in $S^3$ obtained by adding a meridian unknot component to $L$. In this paper, we construct a spectral sequence with the second page isomorphic to the annular Khovanov homology of $L$ and it converges to the reduced Khovanov homology of $\widetilde{L}$. As an application, we classify all the links with the minimal rank of annular Khovanov homology. We also give a proof that annular Khovanov homology detects unlinks.

\end{abstract}

\maketitle

%\tableofcontents

\section{Introduction}\label{section1}

Khovanov \cite{Khovanov1999ACO} defined an invariant for links which assigns a bigraded abelian group $\Kh(L)$ for each link $L\subset S^3$. It is a categorification of Jones polynomial in the sense that it replaces terms in Jones polynomial by graded abelian groups. Since then, many related invariants have been studied, including Lee's deformation and Rasmussen's $s$-invariant \cite{LEE2005554,Rasmussen2004KhovanovHA}, the reduced version \cite{Khovanov2003PatternsIK}, the thickened surface version \cite{Asaeda2004CategorificationOT}, the tangle invariant \cite{bn05} and Khovanov--Rozansky homology \cite{10.2140/gt.2008.12.1387}, etc..

Several spectral sequences that reveal the relationship between Khovanov homology theories and Floer theories have been established. The first one is due to Ozsv\'ath and Szab\'o \cite{Oz03} that builds a connection between the reduced Khovanov homology of the mirror of a link $L$ and the Heegaard Floer homology of the branched double cover of $S^3$ over $L$. Kronheimer and Mrowka \cite{Kronheimer2010KhovanovHI} constructed a spectral sequence with the $E_1$ term isomorphic to Khovanov homology and converging to a version of singular instanton Floer homology.

Let $A$ be an annulus (sometimes it is convenient to view $A$ as a punctured disk). Then the theory of thickened surface \cite{Asaeda2004CategorificationOT} applies for $A\times I$, which is called \textit{annular Khovanov homology}. Roberts \cite{Roberts2013OnKF} constructed a spectral sequence from annular Khovanov homology to Heegaard Floer homology. The analogue of Rasmussen's $s$-invariant in the annular settings was studied \cite{AKhL17}. Xie \cite{XIE2021107864} introduced annular instanton Floer homology for annular links as an analogue of the annular Khovanov homology, and they are also related by a spectral sequence, which can be used to distinguish braids from other tangles \cite{XIE2021107864,XZ19}. 

The relationship between annular Khovanov homology and the original Khovanov homology was studied. There is a natural spectral sequence between them given by ignoring the punctured point \cite[Lemma 2.3]{Roberts2013OnKF}. However, considering the augmentation of links is more helpful to preserve the information about the punctured point.

\begin{defn}
	Let $L\subset A\times I$ be an annular link. The \textit{augmentation} of $L$ is a pointed link $(\widetilde{L},p)\subset \RR^3$ obtained as follows. We view the thickened annulus $A\times I$ as a solid torus in $\RR^3$, and $\widetilde{L}$ is given by the union of $L$ and a meridian circle of $A$ (sometimes we call it an \textit{augmenting circle}). The base point $p$ is chosen on the augmenting circle.
\end{defn}

\begin{figure}[hbtp]
	\centering
	\begin{tikzpicture}
		\draw[dotted] (-2,0) circle [radius=0.175];
		\draw[dotted] (-2,0) circle [radius=1.5];
		\begin{knot}[ignore endpoint intersections=false,clip width=5,clip radius=6pt,looseness=1.2,xshift=-2cm]
			
			\strand[thick](0.5,0) to[out=90,in=0] (0.125,0.375) to[out=180,in=90] (-0.25,0) to[out=-90,in=0] (-0.5,-0.25) to[out=180,in=-90] (-0.75,0);
			
			\strand[thick] (-1,0) to[out=90,in=180] (-0.75,0.25) to[out=0,in=90] (-0.5,0) to[out=-90,in=180] (0,-0.5) to[out=0,in=-90] (0.5,0);
			
			\strand[thick](-0.75,0) to[out=90,in=180] (0.125,0.875) to[out=0,in=90] (1,0) to[out=-90,in=0] (0,-1) to[out=180,in=-90] (-1,0);
		\end{knot}
		
		\begin{knot}[ignore endpoint intersections=false,clip width=5,clip radius=2pt,looseness=1.2,xshift=2cm]
			
			\draw[fill=blue,blue] (0,0) circle[radius=0.04];
			
			\strand[thick](0.5,0) to[out=90,in=0] (0.125,0.375) to[out=180,in=90] (-0.25,0) to[out=-90,in=0] (-0.5,-0.25) to[out=180,in=-90] (-0.75,0);
			
			\strand[thick] (-1,0) to[out=90,in=180] (-0.75,0.25) to[out=0,in=90] (-0.5,0) to[out=-90,in=180] (0,-0.5) to[out=0,in=-90] (0.5,0);
			
			\strand[thick](-0.75,0) to[out=90,in=180] (0.125,0.875) to[out=0,in=90] (1,0) to[out=-90,in=0] (0,-1) to[out=180,in=-90] (-1,0);
			
			\strand[thick,blue](0,0) to[out=-90,in=180] (0.625,-0.2) to[out=0, in=-90](1.25,0) to[out=90,in=0] (0.625,0.2) to[out=180,in=90](0,0);
			
			\flipcrossings{7,9};
		\end{knot}
	\end{tikzpicture}
	\caption{An annular link and its augmentation.}
	\label{whitehead}
\end{figure}
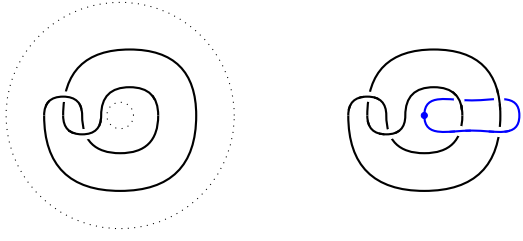

Under this convention, Xie \cite[Section 4.3]{XIE2021107864} showed that the annular instanton Floer homology $\AHI(L)$ is isomorphic to $\II^\natural(\widetilde{L})$, the reduced singular instanton Floer homology of the augmented link. In this paper, we prove the following theorem as an analogue of Xie's result in the Khovanov side. To avoid the sign issues, all the coefficient rings will be $\ZZ$ unless otherwise specified.

\begin{thm}\label{mainthm}
	Let $L\subset A\times I$ be an annular link and let $(\widetilde{L},p)\subset S^3$ be the corresponded augmented link of $L$. Then there is a spectral sequence with the $E_2$ term isomorphic to the annular Khovanov homology $\AKh(L)$ and it converges to the reduced Khovanov homology $\Khr(\widetilde{L},p)$.
\end{thm}

We immediately obtain the following rank inequality.

\begin{cor}\label{rankineq}
	Given an annular link $L$ and its augmentation $\widetilde{L}$, we have \[\rankk_{\ZZ} \AKh(L)\ge \rankk_{\ZZ} \Khr(\widetilde{L},p).\]
\end{cor}

\begin{quest}
	For what link $L$ we have $\AKh(L)$ is isomorphic to $\Khr(\widetilde{L},p)$?
\end{quest}

Theorem \ref{mainthm} provides an alternative way to prove some detection results by referring to the parallel consequences in reduced Khovanov homology. For a link $L$ with $n$ components, it is well-known that $\rankk_{\ZZ}\Khr(L,p)\ge 2^{n-1}$. Hence by the previous corollary, for an annular link $L$, we have \[\rankk_{\ZZ}\AKh(L)\ge \rankk_{\ZZ}\Khr(\widetilde{L},p)\ge 2^n.\]On the other hand, links of minimal rank in $A\times I$ can be classified following \cite{XZ19b}. Before state the result, we first explain the notation. Given a forest $G$, its corresponding link $L_G$ is defined by assigning each vertex of $G$ an unknot component and linking two unknots in the way of Hopf links whenever their corresponding vertices are adjacent. For annular links, the only additional rule is that we need to assign which vertex is corresponding to a nontrivial circle. We say such vertices are \textit{annular} for convenience.

\begin{thm}\label{classify}
	Let $L$ be an $n$-component annular link. Then $\rankk_{\ZZ}\AKh(L)=2^n$ if and only if $L$ is a forest of unknots such that each connected component of the corresponding graph of $L$ contains at most one annular vertex.
\end{thm}

 We say an annular link $U$ is an \textit{unlink} if it has a link diagram $D$ without any crossing. Notice that our definition given here is slightly different to \cite{XIE2021107864}. The following result is a generalization of \cite[Corollary 1.4]{XZ19}.

\begin{cor}\label{detect}
	Let $L$ be an annular link with $n$ components and let $U$ be an annular unlink with $n$ components (might be trivial or nontrivial). Assume that \[\AKh(L)\cong\AKh(U)\]as bigraded (by homological and Alexander gradings) abelian groups. Then $L$ is isotopic to $U$.
\end{cor}

The paper is organized as follows. In Section 2 we review the construction and properties of Khovanov homology. After some preparation in Section 3, we prove Theorem \ref{mainthm} in the last section and discuss its applications.

\subsection*{Acknowledgement}The author would like to thank his advisor Yi Xie for introducing this problem to him and giving him patient and accurate guidance. The author is also grateful to Qing Lan and Xiangqian Yang for helpful conversations. This paper is part of the author's undergraduate research and is partially supported by the elite undergraduate training program of School of Mathematical Sciences, Peking University. 

\section{Review on Khovanov homology theories}\label{sec2}

In this section, we review the construction and properties of the reduced version and the annular version of Khovanov homology.

\subsection{Reduced Khovanov homology}

The reduced version of Khovanov homology is defined in \cite{Khovanov2003PatternsIK} as a categorification of the (normalized) Jones polynomial. We first recall the definition of the original Khovanov homology.

For a link diagram $D$ with $n$ crossings, denote the number of right-handed (resp. left-handed) crossings of $D$ by $n_+$ (resp. $n_-$). For a crossing of $D$, we can use the $0$-smoothing or $1$-smoothing to resolve it, as shown in Figure \ref{01smoothing}. Fix an order of crossings and we can then use vectors $\bfv\in\{0,1\}^n$ to encode resolutions of $D$. Denote the resolution indicated by $\bfv$ by $D_\bfv$, and denote $|\bfv|$ to be the number of $1$-smoothings in $D_\bfv$. Two resolutions that only have difference on one smoothing of crossings are related by a cobordism. The resolutions of $D$ are disjoint unions of circles, and the cobordisms are the merging or splitting of circles. 

\begin{figure}[hbtp]
	\centering
	\begin{tikzpicture}
		
		\begin{knot}[ignore endpoint intersections=false,clip width=5,clip radius=6pt,looseness=1.2]
			
			\strand[thick] (-0.5,0.5)--(0.5,-0.5);
			\strand[thick] (0.5,0.5)--(-0.5,-0.5);
			
		\end{knot}
		\node[below] at (0,-0.7) {a crossing};
		\draw[->,snake=snake] (-0.8,0)--(-1.6,0);
		\draw[->,snake=snake] (0.8,0)--(1.6,0);
		
		\draw[thick] (-3,0.5)  to [out=315,in=180]  (-2.5,0.2) to [out=0,in=225]   (-2,0.5);
		\draw[thick] (-3,-0.5)  to [out=45,in=180]  (-2.5,-0.2) to [out=0,in=135]   (-2,-0.5);  \node[below] at (-2.5,-0.7) {$0$-smoothing};
		
		\draw[thick] (2,0.5)  to [out=315,in=90]  (2.2,0) to [out=270,in=45]    (2,-0.5);  \node[below] at (2.5,-0.7) {$1$-smoothing};
		\draw[thick] (3,0.5)  to [out=225,in=90]  (2.8,0) to [out=270,in=135]   (3,-0.5);
	\end{tikzpicture}
	\caption{Two types of smoothings.}\label{01smoothing}
\end{figure}
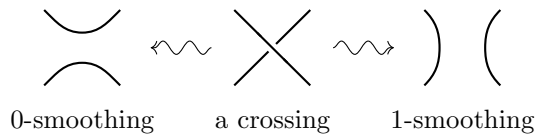

In the case of original Khovanov homology, we apply a $(1+1)$d TQFT to the resolution cube to obtain a chain complex by assigning each circle a graded free abelian group $V:=\ZZ\{v_+,v_-\}$. The resulted complex has two gradings: the homological one and the quantum one, and the latter is specified by $q\deg v_\pm=\pm 1$. Following \cite{BN02}, we denote the shift on these two gradings by $[\bullet]$ and $\{\bullet\}$, respectively. We then take a shift on the quantum grading of chain groups by $|\bfv|$  to ensure the differential preserves the quantum grading and a global shift $[-n_-]\{n_+-2n_-\}$ to ensure the invariance under Reidemeister moves. We finally take cohomology on the chain complex $(\operatorname{CKh}(L),d)$ to obtain $\Kh(L)$.  

To define the reduced version of Khovanov homology, as in other reduced theories, we need to choose a base point $p$ on the link $L$. Every resolution of $L$ has exactly one circle containing $p$, and the generators that take $v_-$ (with the $q$-grading omitted) on this pointed circle span a subcomplex $\CKhr(L,p)\subset \operatorname{CKh}(L)$. The \textit{reduced Khovanov homology} $\Khr(L,p)$ is then defined by the cohomology of $\CKhr(L,p)$. The base point is sometime omitted in the notation if it is clear from the text (e.g. when we are considering an augmented link). As an example, for Hopf link $H$ with a positive linking number, we have \[ \Khr(H,p)=(\ZZ)^{(0,1)}\oplus(\ZZ)^{(2,5)}. \] In general, the following proposition describe the effect on Khovanov homology of making a connected sum with a Hopf link. 

\begin{prop}[{\cite[Theorem 6.1]{AP04}}]\label{hopfadd}
	Let $L$ be a pointed link and let $H$ be the Hopf link with a positive linking number. Then we have a short exact sequence:\[0\to\Khr^{i-1,j-2}(L)\overset{\alpha_*}{\to}\Khr^{i+1,j+3}(L\# H)\overset{\beta_*}{\to}\Khr^{i+1,j+2}(L)\to 0.\]Here $\alpha_*$ and $\beta_*$ are given on a state $S$ as in Figure \ref{ab*}.
	
	\begin{figure}[hbtp]
	\begin{tikzpicture}
		
		\begin{knot}[ignore endpoint intersections=false,clip width=5,clip radius=6pt,looseness=1.2,xshift=-6.5cm]
			\draw[thick] (-0.3,-0.3) rectangle (0.3,0.3); \node at (0,0) {$S$};
			
			\strand[thick] (0,0.3) to[out=up,in=left] (0.3,0.6) to[out=right,in=up] (0.6,0) to[out=down,in=right] (0.3,-0.6) to[out=left,in=down] (0,-0.3);
		\end{knot}
		
		\draw[|->] (-5.5,0)--(-4.5,0) ;\node at (-5,0.2) {$\alpha_*$};
		
		\begin{knot}[ignore endpoint intersections=false,clip width=5,clip radius=6pt,looseness=1.2,xshift=-4cm]
			\draw[thick] (-0.3,-0.3) rectangle (0.3,0.3); \node at (0,0) {$S$};
			
			\strand[thick] (0,0.3) to[out=up,in=left] (0.3,0.6) to[out=right,in=up] (1.5,0) to[out=down,in=right] (0.3,-0.6) to[out=left,in=down] (0,-0.3);
			
			\draw[thick] (0.7,0) circle [radius=0.2] ;\node at (1.1,0.2) {$v_+$};
		\end{knot}	
		
		\begin{knot}[ignore endpoint intersections=false,clip width=5,clip radius=6pt,looseness=1.2,xshift=-1cm]
			\draw[thick] (-0.3,-0.3) rectangle (0.3,0.3); \node at (0,0) {$L$};
			
			\strand[thick] (0,0.3) to[out=up,in=left] (0.3,0.6) to[out=right,in=up] (1,0) to[out=down,in=right] (0.3,-0.6) to[out=left,in=down] (0,-0.3);
			
			\strand[thick] (0.5,0) to[out=up,in=left] (1,0.2) to[out=right,in=up] (1.5,0) to[out=down,in=right] (1,-0.2) to[out=left,in=down] (0.5,0);
			
			\flipcrossings{1};
		\end{knot}	
		
		\begin{knot}[ignore endpoint intersections=false,clip width=5,clip radius=6pt,looseness=1.2,xshift=1.5cm]
			\draw[thick] (-0.3,-0.3) rectangle (0.3,0.3); \node at (0,0) {$S$};
			
			\strand[thick] (0,0.3) to[out=up,in=left] (0.3,0.6) to[out=right,in=up] (0.6,0) to[out=down,in=right] (0.3,-0.6) to[out=left,in=down] (0,-0.3);
		\end{knot}
		
		\draw[|->] (2.5,0)--(3.5,0) ;\node at (3,0.2) {$\beta_*$};
		
		\begin{knot}[ignore endpoint intersections=false,clip width=5,clip radius=6pt,looseness=1.2,xshift=4cm]
			\draw[thick] (-0.3,-0.3) rectangle (0.3,0.3); \node at (0,0) {$S$};
			
			\strand[thick] (0,0.3) to[out=up,in=left] (0.3,0.6) to[out=right,in=up] (0.6,0) to[out=down,in=right] (0.3,-0.6) to[out=left,in=down] (0,-0.3);
			
			\draw[thick] (1,0) circle [radius=0.2] ;\node at (1.4,0.2) {$v_-$};
		\end{knot}
	\end{tikzpicture}
	\caption{The map $\alpha_*$ and $\beta_*$.}
	\label{ab*}
	\end{figure}
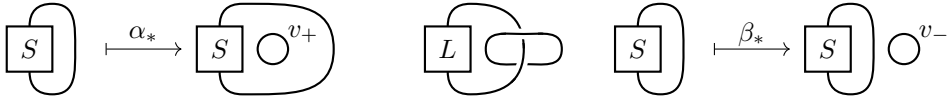
\end{prop}

\subsection{Annular Khovanov homology}The annular version of Khovanov homology can be viewed as a special case of the link homology for links in thickened surfaces defined in \cite{Asaeda2004CategorificationOT}. Let $A$ be an annulus. The annular Khovanov homology assigns a triply-graded abelian group $\AKh(L)$ for each annular link $L\subset A\times I$. We follow the process and notation of \cite{XIE2021107864}. 

Let $D$ be a link diagram of $L$ and define $n,n_\pm,\bfv,D_\bfv,V$ as in the previous subsection. In the annular case, there might be two types of circles in a resolution:  circles that bound disks and circles with nontrivial homologies. We call the first type of circles \textit{trivial} and the second ones \textit{nontrivial}. To obtain the chain groups, we assign $V$ to trivial circles and assign $W:=\ZZ\{w_+,w_-\}$ to nontrivial circles. The differentials are specified by the map corresponding to the merging or splitting of circles, as follows.

\begin{itemize}
		\item Two trivial circles merge into a trivial circle, or one trivial circle splits into two trivial circles. In these cases, the maps are given as same as in Khovanov's original TQFT.
	
		\item One trivial circle and one nontrivial circle merge into a nontrivial circle. In this case, the maps are given by\[v_+\otimes w_\pm\mapsto w_\pm,\ v_-\otimes w_\pm\mapsto 0.\]
		
		\item One nontrivial circle splits into a trivial circle and a nontrivial circle. In this case, the maps are given by \[w_\pm\mapsto v_-\otimes w_\pm.\]
		
		\item Two nontrivial circles merge into a trivial circle. In this case, the maps are given by\[w_\pm\otimes w_\pm\mapsto 0,\ w_\pm\otimes w_\mp\mapsto v_-.\]
		
		\item One trivial circle splits into two nontrivial circles. In this case, the maps are given by\[v_+\mapsto w_+\otimes w_-+w_-\otimes w_+,\ v_-\mapsto 0.\]
\end{itemize}

The homological and quantum grading are given as same as the original case with the additional request that $q\deg w_\pm=\pm 1$. After appropriate shifts, the differential is still filtered of degree $(1,0)$. 

There is the third grading on the chain complex, so-called the \textit{Alexander grading} or $f$-grading, which is specified by $f\deg v_\pm=0$ and $f\deg w_\pm=\pm 1$. The differential preserves the $f$-grading and hence it descends onto the cohomology groups $\AKh(L)$, the annular Khovanov homology. 

\begin{thm}[\cite{Asaeda2004CategorificationOT}]
	The annular Khovanov homology $\AKh(L)$ is an invariant of links in the sense that it is independent of the choice of link diagrams and the order of crossings.
\end{thm}

We conclude this section by some additional remarks. Sometimes we write $\AKh(L,m)$ to indicate the $f$-degree $m$ summand of $\AKh(L)$. If $L$ is contained in a ball $B^3\subset A\times I$, then $\AKh(L)$ is supported on $f=0$ and $\AKh(L)\cong \Kh({L})$. Both the reduced Khovanov homology and the annular Khovanov homology are functorial. That is, a cobordism $\rho\colon L_1\to L_2$ between links (resp. annular links) induces a (filtered) map between Khovanov homology groups \[ \Khr(\rho)\colon\Khr(L_1)\to\Khr(L_2)\ (\text{resp. }\AKh(\rho)\colon\AKh(L_1)\to\AKh(L_2)).\] 

\section{The unlink case}\label{sec3}

In this section, we construct an isomorphism between the annular Khovanov homology of an annular unlink and the reduced Khovanov homology of its augmentation. We show that such an isomorphism is compatible with the group homomorphisms induced by the cobordism maps.

\subsection{Homology groups}

Denote the annular unlink with $n$ nontrivial unknot components by $U_n$ and let $\widetilde{U_n}$ be its augmentation. In the language of \cite{XZ19b}, $\widetilde{U_n}$ corresponds to the graph shown in Figure \ref{treen}.

	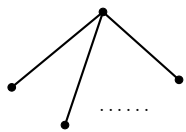
\begin{figure}[hbtp]
	\begin{tikzpicture}
		\draw[fill=black] (0,1) circle[radius=0.05];
		\draw[fill=black] (-1.2,0) circle[radius=0.05];
		\draw[fill=black] (-0.5,-0.5) circle[radius=0.05];
		\draw[fill=black] (1,0.1) circle[radius=0.05];
		\draw[thick] (0,1)--(-1.2,0);
		\draw[thick] (0,1)--(-0.5,-0.5);
		\draw[thick] (0,1)--(1,0.1);

		\node[scale=0.75] at (0.3,-0.3) {$\dots\dots$};
	\end{tikzpicture}
	\caption{The tree corresponding to $\widetilde{U_n}$.}
	\label{treen}
	\end{figure}

The obvious diagram of $U_n$ contains $n$ disjoint nontrivial circles. In this section, we will stick on this diagram to calculate homology groups. We assign the number $1$ to $n$ from the innermost nontrivial circle to the outermost one. By Proposition \ref{hopfadd}, the Poincar\'e polynomial of $\Khr(\widetilde{U_n})$ is given by \[P(\widetilde{U_n})=(tq^3)^n(tq^2+t^{-1}q^{-2})^n.\]Here the homological and quantum grading are indicated by $t$ and $q$ respectively.

Each original component of $\widetilde{U_n}$ has two crossings with the meridian circle. There are $2^n$ resolutions such that every pair of crossings is resolved by the same smoothing. We say such resolutions are \textit{symmetric} and encode them by $0-1$ sequences of length $n$, as shown in Figure \ref{reso10}. Notice that a symmetric resolution always has $n$ (unpointed) components. We denote the cobordism of changing one crossing (on the $k$-th strand) from $0$-smoothing to $1$-smoothing by $(\cdots \bullet\cdots)$ (here the mark $\bullet$ is on the $k$-th digit).

\begin{figure}[hbtp]
	\begin{tikzpicture}
		\begin{knot}[scale=1.5,ignore endpoint intersections=false,clip width=5,clip radius=6pt,looseness=1.2,xshift=0cm,yshift=0cm]
			
			\strand[thick] (0,0) circle[radius=0.4];
			\strand[thick] (0,0) circle[radius=0.2];
			\strand[thick] (0,0) to[out=up,in=left] (0.2,0.2) to[out=right,in=up] (0.6,0) to[out=down,in=right] (0.2,-0.2) to[out=left,in=down] (0,0);
			\flipcrossings{2,4};
			
		\end{knot}
		\draw[fill=black] (0,0) circle[radius=0.06];
		\draw[scale=1.5,thick,green] (0.1,0.1)--(0.1,0.25);
		\draw[scale=1.5,thick,green] (0.1,-0.1)--(0.1,-0.25);
		\draw[scale=1.5,thick,green] (0.42,0.15)--(0.25,0.25);
		\draw[scale=1.5,thick,green] (0.42,-0.15)--(0.25,-0.25);
	\end{tikzpicture}
	\caption{The symmetric resolution $(10)$ of $\widetilde{U_2}$.}
	\label{reso10}
\end{figure}
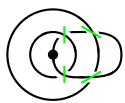

We can now describe the generators of $\Khr(\widetilde{U_n})$ explicitly. 

\begin{prop}\label{hopfsym}
	For each symmetric resolution $v\in\{0,1\}^n$, we can choose an element $e_v$ lying in the chain group corresponding to this resolution. The collection of $e_v$ descends to a generating set of $\Khr(\widetilde{U_n})$.
\end{prop}

\begin{proof}
	We prove the proposition by induction. There is nothing to say for $n=0$. In general, by applying proposition \ref{hopfadd} to $L=\widetilde{U_{n-1}}$ and $L\#H=\widetilde{U_n}$, we obtain a short exact sequence \[0\to\Khr^{i-1,j-2}(\widetilde{U_{n-1}})\overset{\alpha_*}{\to}\Khr^{i+1,j+3}(\widetilde{U_n})\overset{\beta_*}{\to}\Khr^{i+1,j+2}(\widetilde{U_{n-1}})\to 0.\]Let $v=(v_1,v_2,\dots,v_n)\in\{0,1\}^n$ and let $v'=(v_1,\dots,v_{n-1})$. The sequence $v'$ corresponds to a symmetric resolution $R'_{v'}$ of $\widetilde{U_{n-1}}$. If $v_n=1$, we just need to take \[ e_v=\alpha_*(e_{v'})=e_{v'}\otimes v_+. \]
	
	If $v_n=0$ and $e_{v'}=A\otimes v_++B\otimes v_-$, here $v_\pm$ are associated to the $(n-1)$-th circle, we take \[e_v=e_{v'}\otimes v_-+A\otimes v_-\otimes v_+.\]It is clear that $\beta_*(e_v)=e_{v'}$ and it remains to show that $e_v$ is a cycle. Notice that the cobordism $(v',\bullet)$ is always a merging (rather than a splitting) of circles, and the construction ensures that $\Khr((v',\bullet))(e_v)=0$. We show that other cobordisms also vanish by discussing the value of $v_{n-1}$, see Figure \ref{vn-1}.
	
	\begin{figure}[hbtp]
		\begin{tikzpicture}
			
			\begin{knot}[scale=2,ignore endpoint intersections=false,clip width=5,clip radius=6pt,looseness=1.2,xshift=-1cm,yshift=0cm]
				
				\strand[thick] (0,0) circle[radius=0.4];
				
				\strand[thick] (0,0) circle[radius=0.32];
				
				\strand[thick] (0,0) to[out=up,in=left] (0.2,0.2) to[out=right,in=up] (0.6,0) to[out=down,in=right] (0.2,-0.2) to[out=left,in=down] (0,0);
				
				\flipcrossings{2,4};
				
				\node at (0,-0.6) {$v_{n-1}=1$};
				\node[scale=0.5] at (0.2,0) {$\dots$};
				\node[scale=0.5] at (-0.2,0) {$\dots$};
				
			\end{knot}
			\draw[fill=black,scale=2,xshift=-1cm] (0,0) circle[radius=0.03];
			\draw[scale=2,thick,green,xshift=-1cm] (0.23,0.15)--(0.23,0.3);
			\draw[scale=2,thick,green,xshift=-1cm] (0.23,-0.13)--(0.23,-0.28);
			\draw[scale=2,thick,green,xshift=-1cm] (0.42,0.15)--(0.27,0.24);
			\draw[scale=2,thick,green,xshift=-1cm] (0.42,-0.15)--(0.27,-0.24);
			
			\begin{knot}[scale=2,ignore endpoint intersections=false,clip width=5,clip radius=6pt,looseness=1.2,xshift=1cm,yshift=0cm]
				
				\strand[thick] (0,0) circle[radius=0.4];
				
				\strand[thick] (0,0) circle[radius=0.32];
				
				\strand[thick] (0,0) to[out=up,in=left] (0.2,0.2) to[out=right,in=up] (0.6,0) to[out=down,in=right] (0.2,-0.2) to[out=left,in=down] (0,0);
				
				\flipcrossings{2,4};
				
				\node at (0,-0.6) {$v_{n-1}=0$};
				\node[scale=0.5] at (0.2,0) {$\dots$};
				\node[scale=0.5] at (-0.2,0) {$\dots$};
				
			\end{knot}
			\draw[fill=black,scale=2,xshift=1cm] (0,0) circle[radius=0.03];
			\draw[scale=2,thick,green,xshift=1cm] (0.32,0.15)--(0.17,0.24);
			\draw[scale=2,thick,green,xshift=1cm] (0.32,-0.15)--(0.17,-0.24);
			\draw[scale=2,thick,green,xshift=1cm] (0.42,0.15)--(0.27,0.24);
			\draw[scale=2,thick,green,xshift=1cm] (0.42,-0.15)--(0.27,-0.24);
		\end{tikzpicture}
		\caption{Possible resolutions with $v_n=0$.}
		\label{vn-1}
	\end{figure}
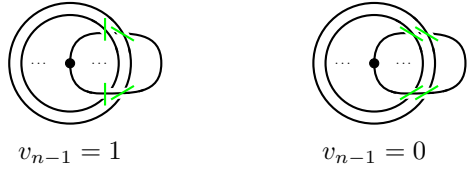

Notice that the cobordism map that the change happens on the $i$-th strand ($1\le i\le n-2$) vanishes on $A,B$. Hence if $v_{n-1}=1$, then there is no possibly non-vanishing cobordism map. Now assume that $v_{n-1}=0$ and let $v''=(v_1,\dots,v_{n-2})$, \[e_{v''}=A_1\otimes v_++B_1\otimes v_-.\] Then we have\[e_{v'}=(A_1\otimes v_++B_1\otimes v_-)\otimes v_-+A_1\otimes v_-\otimes v_+,\]\[e_v=((A_1\otimes v_++B_1\otimes v_-)\otimes v_-+A_1\otimes v_-\otimes v_+)\otimes v_-+A_1\otimes v_-\otimes v_-\otimes v_+,\]and hence $\Khr((v'',\bullet,1))(e_v)=0$.
\end{proof}

We now construct an explicit identification between $\AKh(U_n)$ and $\Khr(\widetilde{U_n})$. On the level of homology, this is quite easy: the Poincar\'e polynomial of $\AKh(U_n)$ is given by \[P(U_n)=(fq+f^{-1}q^{-1})^n.\]Here the $f$-grading is indicated by $f$. The substitution $f\mapsto tq$ gives an isomorphism between $\AKh(U_n)$ and $\Khr(U_n')$ (up to shifting). More concretely, the generator \[w=w^{(1)}_\pm\otimes w^{(2)}_\pm\otimes\dots\otimes w^{(n)}_\pm\in \AKh(U_n)\] is identified with the generator corresponding to the symmetric resolution of label $(v_1,v_2,\dots,v_n)$, where $v_i=1$ if and only if $w^{(i)}_+$ appears in $w$ ($i=1,2,\dots,n$), as in Proposition \ref{hopfsym}.  

The effect of adding a trivial unknot component to $U_n$ is just taking two copies of the original homology groups with generators tensoring with $v_\pm$ respectively, by K\"unneth formula. We summarize the consequence of this subsection in the following form.

\begin{thm}\label{idtphi}
	Let $L$ be an annular unlink with $n$ nontrivial unknot components, and let $\widetilde{L}$ be its augmentation. Then there is an isomorphism $\Phi_L$ between the annular Khovanov homology of $L$ and the reduced Khovanov homology of $\widetilde{L}$. More precisely, we have an isomorphism \[\Phi_L\colon\AKh^{i,j,k}(L) \to\Khr^{i+k+n,j+k+3n}(\widetilde{L}).\] The correspondence of generators is given above.\qed
\end{thm}

\subsection{Functority}

A cobordism between annular links naturally induces a cobordism between their augmentations. In this subsection, we show that the isomorphism $\Phi_L$ defined in Theorem \ref{idtphi} is compatible with cobordisms. According to our purpose (see the next section), we don't need to deal with the Reidemeister moves on the diagram of $L$, and we concentrate on Morse moves, i.e. the merging and splitting of circles. We first verify the compatibility with only related circles and then consider the effect of adding other unlink components. There are four cases we need to discuss:

\begin{enumerate}[a.]
	\item one trivial circle and one nontrivial circle merge into a nontrivial circle;
	
	\item one nontrivial circle splits into a trivial circle and a nontrivial circle;
	
	\item two nontrivial circles merge into a trivial circle;
	
	\item one trivial circle splits into two nontrivial circles.
\end{enumerate}

Since the homomorphisms induced by cobordisms are well-defined \cite{Jacobsson2004AnIO}, we may choose specific link diagrams to calculate them. Case a. and b. are simple diagram chasing. The following diagram illustrates this process.

	\begin{figure}[hbtp]
	\begin{tikzpicture}
		
		\draw[->] (-1,2)--(1,2) ;\node at (0,2.25) {$\Phi_{L_1}$};
		\draw[->] (-1,0)--(1,0) ;\node at (0,-0.25) {$\Phi_{L_2}$};
		\draw[->] (-3,1.5)--(-3,0.5) ;
		\draw[->] (3,1.5)--(3,0.5) ;
		
		\node at (-4,2) {$\AKh($};
		\begin{knot}[ignore endpoint intersections=false,clip width=5,clip radius=6pt,looseness=1.2,xshift=-3cm,yshift=2cm]
			\strand[thick] (-0.25,0) circle[radius=0.3];
			
			\strand[thick] (0.5,0) circle[radius=0.3];
			
			\draw[fill=black] (0.5,0) circle[radius=0.04];

		\end{knot}
		\node at (-2.1,2) {$)$};

		\node at (2,2) {$\Khr($};
		\begin{knot}[ignore endpoint intersections=false,clip width=5,clip radius=6pt,looseness=1.2,xshift=3cm,yshift=2cm]
			\strand[thick] (-0.25,0) circle[radius=0.3];
			\strand[thick] (0.5,0) circle[radius=0.3];
			\strand[thick] (0.5,0) to[out=up,in=left] (0.7,0.2) to[out=right,in=up] (1.1,0) to[out=down,in=right] (0.7,-0.2) to[out=left,in=down] (0.5,0);
			
			\draw[fill=black] (0.5,0) circle[radius=0.04];
			
			\flipcrossings{2};
		\end{knot}
		\node at (4.2,2) {$)$};
		
		\node at (-4,0) {$\AKh($};
		\begin{knot}[ignore endpoint intersections=false,clip width=5,clip radius=6pt,looseness=1.2,xshift=-3cm,yshift=0cm]

			\draw[thick] (0,0.15) arc (30:330:0.3);
			
			\draw[thick] (0.25,0.15) arc (150:-150:0.3);
			
			\strand[thick] (0,0.15) to[out=-60,in=-120] (0.25,0.15);
			
			\strand[thick] (0,-0.15) to[out=60,in=120] (0.25,-0.15);
			
			\draw[fill=black] (0.5,0) circle[radius=0.04];
		\end{knot}
		\node at (-2.1,0) {$)$};

		\node at (2,0) {$\Khr($};
		\begin{knot}[ignore endpoint intersections=false,clip width=5,clip radius=6pt,looseness=1.2,xshift=3cm,yshift=0cm]
			\draw[thick] (0,0.15) arc (30:330:0.3);

			\strand[thick] (0.25,0.15) to[out=60,in=left] (0.5,0.3) to[out=right,in=up] (0.8,0) to[out=down,in=right] (0.5,-0.3) to[out=left,in=-60] (0.25,-0.15);
			
			\strand[thick] (0,0.15) to[out=-60,in=-120] (0.25,0.15);
			
			\strand[thick] (0,-0.15) to[out=60,in=120] (0.25,-0.15);
			
			\strand[thick] (0.5,0) to[out=up,in=left] (0.7,0.2) to[out=right,in=up] (1.1,0) to[out=down,in=right] (0.7,-0.2) to[out=left,in=down] (0.5,0);
			
			\draw[fill=black] (0.5,0) circle[radius=0.04];
			
			\flipcrossings{4};
		\end{knot}
		\node at (4.2,0) {$)$};
		
		\node at (-1.5,1.5) {$v_-\otimes w_+$};
		
		\node at (-1.5,0.5) {$0$};
		
		\node at (1.5,1.5) {$v_-\otimes v_+$};
		
		\node at (1.5,0.5) {$0$};
		
		\draw[|->] (-0.75,1.5)--(0.75,1.5);
		\draw[|->] (-0.75,0.5)--(0.75,0.5);
		\draw[|->] (-1.5,1.25)--(-1.5,0.75);
		\draw[|->] (1.5,1.25)--(1.5,0.75);
		
	\end{tikzpicture}
	\caption{Case a..}
	\label{casea}
	\end{figure}

In the case c. and d., we need to check the following diagrams commute.

	\begin{figure}[hbtp]
	\centering
	\begin{tikzpicture}
		
		\draw[->] (-5.5,1.5)--(-4.5,1.5) ;\node at (-5,1.75) {$\Phi_{L_3}$};
		\draw[->] (-5.5,0)--(-4.5,0) ;\node at (-5,-0.25) {$\Phi_{L_4}$};
		\draw[->] (-7,1)--(-7,0.5) ;
		\draw[->] (-3,1)--(-3,0.5) ;
		
		\node at (-7.5,1.5) {$\AKh($};
		\begin{knot}[ignore endpoint intersections=false,clip width=5,clip radius=6pt,looseness=1.2,xshift=-6.5cm,yshift=1.5cm]
			\strand[thick] (0,0) circle[radius=0.4];
			
			\strand[thick] (0,0) circle[radius=0.2];

		\end{knot}
		\node at (-5.8,1.5) {$)$};
		\node at (-4,1.5) {$\Khr($};

		\begin{knot}[ignore endpoint intersections=false,clip width=5,clip radius=6pt,looseness=1.2,xshift=-3cm,yshift=1.5cm]
			\strand[thick] (0,0) circle[radius=0.4];
			
			\strand[thick] (0,0) circle[radius=0.2];
			
			\strand[thick] (0,0) to[out=up,in=left] (0.2,0.2) to[out=right,in=up] (0.6,0) to[out=down,in=right] (0.2,-0.2) to[out=left,in=down] (0,0);
			
			\flipcrossings{2,4};
			
		\end{knot}
		\node at (-2,1.5) {$)$};
		\draw[fill=black] (-3,1.5) circle[radius=0.04];
		
		\node at (-7.5,0) {$\AKh($};
		\begin{knot}[ignore endpoint intersections=false,clip width=5,clip radius=6pt,looseness=1.2,xshift=-6.5cm,yshift=0cm]
			\strand[thick,domain=150:-150] plot ({0.4*cos(\x)}, {0.4*sin(\x)});
			
			\strand[thick,domain=150:-150] plot ({0.2*cos(\x)}, {0.2*sin(\x)});
			
			\draw[thick] (150:0.4) to[out=240,in=240] (150:0.2);
			
			\draw[thick] (210:0.4) to[out=120,in=120] (210:0.2);
			
			\draw[fill=black] (0,0) circle[radius=0.04];
		\end{knot}
		\node at (-5.8,0) {$)$};
		\node at (-4,0) {$\Khr($};
		\begin{knot}[ignore endpoint intersections=false,clip width=5,clip radius=6pt,looseness=1.2,xshift=-3cm,yshift=0cm]
			\strand[thick,domain=150:-150] plot ({0.4*cos(\x)}, {0.4*sin(\x)});
			
			\strand[thick,domain=150:-150] plot ({0.2*cos(\x)}, {0.2*sin(\x)});
			
			\draw[thick] (150:0.4) to[out=240,in=240] (150:0.2);
			
			\draw[thick] (210:0.4) to[out=120,in=120] (210:0.2);
			
			\strand[thick] (0,0) to[out=up,in=left] (0.2,0.2) to[out=right,in=up] (0.6,0) to[out=down,in=right] (0.2,-0.2) to[out=left,in=down] (0,0);
			
			\flipcrossings{2,4};
		\end{knot}
		\node at (-2,0) {$)$};
		\draw[fill=black] (-3,0) circle[radius=0.04];
		\draw[->] (1,1.5)--(2,1.5) ;\node at (1.5,1.75) {$\Phi_{L_4}$};
		\draw[->] (1,0)--(2,0) ;\node at (1.5,-0.25) {$\Phi_{L_3}$};
		\draw[->] (-0.5,1)--(-0.5,0.5) ;
		\draw[->] (3.5,1)--(3.5,0.5) ;
		
		\node at (-1,1.5) {$\AKh($};
		\begin{knot}[ignore endpoint intersections=false,clip width=5,clip radius=6pt,looseness=1.2,xshift=0cm,yshift=0cm]
			\strand[thick] (0,0) circle[radius=0.4];
			
			\strand[thick] (0,0) circle[radius=0.2];
			
			\draw[fill=black] (0,0) circle[radius=0.04];

		\end{knot}
		\node at (0.7,1.5) {$)$};
		\node at (2.5,1.5) {$\Khr($};
		\begin{knot}[ignore endpoint intersections=false,clip width=5,clip radius=6pt,looseness=1.2,xshift=3.5cm,yshift=0cm]
			\strand[thick] (0,0) circle[radius=0.4];
			
			\strand[thick] (0,0) circle[radius=0.2];
			
			\strand[thick] (0,0) to[out=up,in=left] (0.2,0.2) to[out=right,in=up] (0.6,0) to[out=down,in=right] (0.2,-0.2) to[out=left,in=down] (0,0);
			
			\flipcrossings{2,4};
			
		\end{knot}
		\node at (4.5,1.5) {$)$};
		\draw[fill=black] (3.5,0) circle[radius=0.04];
		\node at (-1,0) {$\AKh($};
		\begin{knot}[ignore endpoint intersections=false,clip width=5,clip radius=6pt,looseness=1.2,xshift=0cm,yshift=1.5cm]
			\strand[thick,domain=150:-150] plot ({0.4*cos(\x)}, {0.4*sin(\x)});
			
			\strand[thick,domain=150:-150] plot ({0.2*cos(\x)}, {0.2*sin(\x)});
			
			\draw[thick] (150:0.4) to[out=240,in=240] (150:0.2);
			
			\draw[thick] (210:0.4) to[out=120,in=120] (210:0.2);
			
			\draw[fill=black] (0,0) circle[radius=0.04];
		\end{knot}
		\node at (0.7,0) {$)$};

		\node at (2.5,0) {$\Khr($};
		\begin{knot}[ignore endpoint intersections=false,clip width=5,clip radius=6pt,looseness=1.2,xshift=3.5cm,yshift=1.5cm]
			\strand[thick,domain=150:-150] plot ({0.4*cos(\x)}, {0.4*sin(\x)});
			
			\strand[thick,domain=150:-150] plot ({0.2*cos(\x)}, {0.2*sin(\x)});
			
			\draw[thick] (150:0.4) to[out=240,in=240] (150:0.2);
			
			\draw[thick] (210:0.4) to[out=120,in=120] (210:0.2);
			
			\strand[thick] (0,0) to[out=up,in=left] (0.2,0.2) to[out=right,in=up] (0.6,0) to[out=down,in=right] (0.2,-0.2) to[out=left,in=down] (0,0);
			
			\flipcrossings{2,4};
		\end{knot}
		\draw[fill=black] (3.5,1.5) circle[radius=0.04];
		\node at (4.5,0) {$)$};
	\end{tikzpicture}
	\caption{Case c. and d..}
	\label{casecd}
\end{figure}

Denote the upper and the lower links in the leftmost column of Diagram \ref{casecd} by $L_3,L_4$, respectively. We have \begin{align*}
	\Khr(\widetilde{L_3})&\cong (\ZZ)^{(0,2)}\oplus((\ZZ)^{(2,6)})^{\oplus 2}\oplus(\ZZ)^{(4,10)},\\
	\Khr(\widetilde{L_4})&\cong (\ZZ)^{(0,1)}\oplus (\ZZ)^{(0,-1)}.
\end{align*}

We first check case c.. Notice that the cobordism map $\Khr(\widetilde{L_3})\to \Khr(\widetilde{L_4})$ is of degree $(-2,-7)$, the only possibly nontrivial map is \[((\ZZ)^{(2,6)})^{\oplus 2}\to (\ZZ)^{(0,-1)},\]which corresponds to the merging map in the leftmost column of Diagram \ref{casecd}:\[w_+\otimes w_-,\ w_-\otimes w_+\mapsto v_-.\]By the algorithm given in Theorem \ref{idtphi}, $w_-\otimes w_+$ and $w_+\otimes w_-$ correspond to $v_-\otimes v_+$ (associated to the symmetric resolution $(01)$) and $v_+\otimes v_-+v_-\otimes v_+$ (associated to the symmetric resolution $(10)$) respectively. Images of them are $v_-\otimes v_-\otimes v_+$ and $v_-$ respectively. It suffices to show they are non-vanishing and cohomologous. Denote the crossing number $1$ to $4$ as in Figure \ref{label1234} 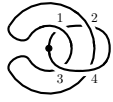
\begin{figure}[hbtp]
\begin{tikzpicture}
	\begin{knot}[ignore endpoint intersections=false,clip width=5,clip radius=6pt,looseness=1.2,xshift=0cm,yshift=0cm]
		\strand[thick,domain=150:-150] plot ({0.6*cos(\x)}, {0.6*sin(\x)});
		
		\strand[thick,domain=150:-150] plot ({0.3*cos(\x)}, {0.3*sin(\x)});
		
		\draw[thick] (150:0.6) to[out=240,in=240] (150:0.3);
		
		\draw[thick] (210:0.6) to[out=120,in=120] (210:0.3);
		
		\strand[thick] (0,0) to[out=up,in=left] (0.4,0.3) to[out=right,in=up] (0.8,0) to[out=down,in=right] (0.4,-0.3) to[out=left,in=down] (0,0);
		
		\flipcrossings{2,4};
	\end{knot}
	\draw[fill=black] (0,0) circle[radius=0.04];
	\node[scale=0.5] at (0.15,0.4) {$1$};
	\node[scale=0.5] at (0.6,0.4) {$2$};
	\node[scale=0.5] at (0.15,-0.4) {$3$};
	\node[scale=0.5] at (0.6,-0.4) {$4$};
\end{tikzpicture}
\caption{The label of crossings on $\widetilde{L_4}$.}
\label{label1234}
\end{figure}and take the lexicographical order on the resolutions (i.e. take the states associated to the resolution $(1100)$ first, then $(1010)$, etc., and the explicit order of basis can be read from texts), and we write down differentials in matrix. Here we denote the bases of chain groups $\operatorname{CKhr}^{(-1,-1)}(\widetilde{L_4})$, $\operatorname{CKhr}^{(0,-1)}(\widetilde{L_4})$ by $e_i$($1\le i\le 6$), $f_j$($1\le j\le 8$), respectively. We have\begin{align*}
	\operatorname{d}^{(-1,-1)}&=\begin{pmatrix}
		1&1&0&0&0&0\\
		1&0&0&1&0&0\\
		1&0&0&0&1&0\\
		0&0&1&1&0&0\\
		0&1&0&0&1&0\\
		0&0&1&0&0&1\\
		0&0&1&0&1&0\\
		0&0&0&1&0&1\\		
	\end{pmatrix},\\
	\operatorname{d}^{(0,-1)}&=\begin{pmatrix}
		1&0&1&0&1&0&0&0\\
		0&0&0&1&0&1&0&1\\
	\end{pmatrix}.
\end{align*}Elements  $v_-$ and $v_-\otimes v_-\otimes v_+$ correspond to the vector $f_2$ and $f_7$, respectively. It is easy to see that $f_2-f_7=\operatorname{d}^{(-1,-1)}(e_1+e_2+e_5)$ and $f_2\notin \operatorname{Im}\operatorname{d}^{(-1,-1)}$. This finishes the verification in case c..

The verification in case d. is essentially the same. The only possibly nontrivial map in the rightmost column of Diagram \ref{casecd} is \[(\ZZ)^{(0,1)}\to ((\ZZ)^{(2,6)})^{\oplus 2},\]which corresponds to the splitting map \[v_+\mapsto w_+\otimes w_-+w_-\otimes w_+\]in the third column of Diagram \ref{casecd}. We take bases of $\operatorname{CKhr}^{(-1,1)}(\widetilde{L_4}),\operatorname{CKhr}^{(0,1)}(\widetilde{L_4})$ as in case c., and we have \begin{align*}
	\operatorname{d}^{(0,1)}&=\begin{pmatrix}
		1&1&0&1&0&0&0&0\\	
		0&1&1&0&0&0&0&1\\	
		1&0&1&0&0&0&1&0\\	
		1&0&0&0&0&1&0&0\\	
		0&0&0&1&0&0&1&1\\	
		0&0&0&0&1&0&0&1\\	
	\end{pmatrix},\\
	\operatorname{d}^{(-1,1)}&=\begin{pmatrix}
		1&0&0&1&0&1&1&0\\
		0&0&1&0&1&0&1&1\\
	\end{pmatrix}^{T}.
\end{align*}The generator of $\Khr^{(0,1)}(\widetilde{L_4})$ can be represented by $v^{(0,1)}=(1,1,0,0,1,1,1,1)^T$, and we have \[\Khr(\widetilde{L_4})\to \Khr(\widetilde{L_3})\colon v^{(0,1)}\mapsto(1,0,0,0,1,0,1,1)^T.\]The boundary subgroup of degree $(2,6)$ is spanned by the image of \[
\operatorname{d}^{(1,6)}=\begin{pmatrix}
1&1&0&0\\
1&0&1&0\\
1&0&0&1\\
0&1&1&0\\
0&0&0&0\\
0&1&0&1\\
0&1&0&1\\
0&0&1&1\\		
\end{pmatrix}.\]Therefore, under the map $\Khr(\widetilde{L_4})\to\Khr( \widetilde{L_3})$, we have\begin{align*}
v^{(0,1)}\mapsto\Phi_{L_3}( w_+\otimes w_-+w_-\otimes w_+)+\operatorname{d}^{(1,6)}(0,1,1,0)^T.
\end{align*}This completes the verification in case d..

It remains to consider the effect of adding a new unlink component to the cobordism. The case of adding a trivial unknot component is trivial and we assume that the additional unknot component is nontrivial. Let $L_1,L_2$ be two annular unlinks and let $\rho\colon L_1\to L_2$ be a cobordism obtained by a Morse move. We have \[\AKh(\rho\coprod\idd)=\AKh(\rho)\otimes\idd_{U},\]here $U=U_1$ is the nontrivial annular unknot. Take $S\in\AKh(L_1)$ and let $T=\AKh(\rho)(S)$. By Proposition \ref{hopfadd} and Theorem \ref{idtphi}, the following diagram commutes. 	
\begin{center}
	\begin{tikzcd}
		S\otimes w_+\arrow[r,mapsto,"\Phi_{L_1}"]\arrow[d,mapsto,"\AKh(\rho)"']& \Phi_{L_1}(S)\otimes v_+\arrow[d,mapsto,"\Khr(\rho')"]\\
		T\otimes w_+\arrow[r,mapsto,"\Phi_{L_2}"']&\Phi_{L_2}(S)\otimes v_+
	\end{tikzcd}.
\end{center}Assume that $\Phi_{L_1}(S)=A\otimes v_++B\otimes v_-$ and $\Phi_{L_2}(T)=C\otimes v_++D\otimes v_-$. By Proposition \ref{hopfadd} and Theorem \ref{idtphi}, the following diagram commutes, which completes the proof. 
\begin{center}
\begin{tikzcd}
	S\otimes w_-\arrow[r,mapsto,"\Phi_{L_1}"]\arrow[d,mapsto,"\AKh(\rho)"']& \Phi_{L_1}(S)\otimes v_-+A\otimes v_-\otimes v_+\arrow[d,mapsto,"\Khr(\widetilde{\rho})"]\\
	T\otimes w_-\arrow[r,mapsto,"\Phi_{L_2}"']&\Phi_{L_2}(T)\otimes v_-+C\otimes v_-\otimes v_+
\end{tikzcd}.
\end{center}

In summary, we have shown the following theorem. Roughly speaking, it gives a natural isomorphism between two cohomology theories on annular unlinks.

\begin{thm}\label{canonicaliso}
	Let $L_1,L_2$ be two annular unlinks and let $\rho\colon L_1\to L_2$ be a cobordism obtained by composition of Morse moves. The cobordism $\rho$ induces a cobordism $\widetilde{\rho}$ between the augmentations $\widetilde{L_1}$ and $\widetilde{L_2}$. Let $\Phi_{L_1},\Phi_{L_2}$ be  isomorphisms given in Theorem \ref{idtphi}. Then the following diagram commutes.
	\begin{center}
		\begin{tikzcd}
			\AKh(L_1)\arrow[r,"\Phi_{L_1}"]\arrow[d,"\AKh(\rho)"']& \Khr(\widetilde{L_1})\arrow[d,"\Khr(\widetilde{\rho})"]\\
			\AKh(L_2)\arrow[r,"\Phi_{L_2}"']&\Khr(\widetilde{L_2})
		\end{tikzcd}.
	\end{center}\qed
\end{thm}

\section{The spectral sequence}\label{sec4}

In this section, we prove Theorem \ref{mainthm} and discuss some examples and applications. To prove Theorem \ref{mainthm}, we choose a link diagram as shown in Figure \ref{standarddgm}. For convenience, we call the strands appearing in the right the \textit{annular strands}.

	\begin{figure}[hbtp]
	\centering
	\begin{tikzpicture}
		\draw[dotted] (0,0) circle [radius=0.175];
		\draw[dotted] (0,0) circle [radius=1.5];
		\begin{knot}[ignore endpoint intersections=false,clip width=5,clip radius=6pt,looseness=1.2]
			\draw[thick] (-1.1,-0.3) rectangle (-0.5, 0.3); \node at (-0.8,0) {$L$};
			
			\strand[thick] (-1,0.3) to[out=up,in=left] (0,1) to[out=right,in=up] (1,0) to[out=down,in=right] (0,-1) to[out=left,in=down] (-1,-0.3);
			
			\draw[densely dotted,thick,scale=0.5] (55:1.3)--(55:2);
			
			\node[scale=0.5] at (45:0.8) {$n$};
			
			\strand[thick] (-0.6,0.3) to[out=up,in=left] (0,0.6) to[out=right,in=up] (0.6,0) to[out=down,in=right] (0,-0.6) to[out=left,in=down] (-0.6,-0.3);
		\end{knot}
		
		\begin{knot}[ignore endpoint intersections=false,clip width=5,clip radius=6pt,looseness=1.2,xshift=5cm]
			\draw[fill=black] (0,0) circle[radius=0.08];
			\strand[thick,black](0,0) to[out=-90,in=180] (0.7,-0.2) to[out=0, in=-90](1.4,0) to[out=90,in=0] (0.7,0.2) to[out=180,in=90](0,0);
			
			\draw[thick] (-1.1,-0.3) rectangle (-0.5, 0.3); \node at (-0.8,0) {$L$};
			
			\strand[thick] (-1,0.3) to[out=up,in=left] (0,1) to[out=right,in=up] (1,0) to[out=down,in=right] (0,-1) to[out=left,in=down] (-1,-0.3);
			
			\draw[densely dotted,thick,scale=0.5] (55:1.3)--(55:2);
			
			\node[scale=0.5] at (45:0.8) {$n$};
			
			\strand[thick] (-0.6,0.3) to[out=up,in=left] (0,0.6) to[out=right,in=up] (0.6,0) to[out=down,in=right] (0,-0.6) to[out=left,in=down] (-0.6,-0.3);
			
			\flipcrossings{2,4};
		\end{knot}
	\end{tikzpicture}
	\caption{A standard link diagram and its augmentation.}
	\label{standarddgm}
\end{figure}
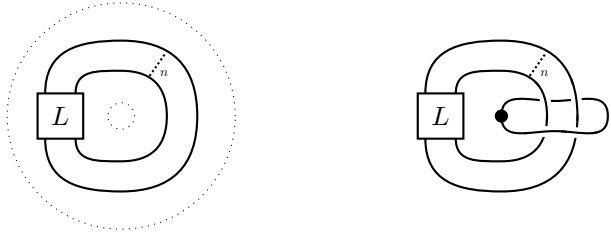

\begin{proof}[Proof of Theorem \ref{mainthm}]
	Fixed a link diagram $D$ as above. Crossings of $\widetilde{L}$ can be classified into two types: crossings of the augmenting circle and the annular strands, and the original crossings of $L$. We encode the resolutions of the first type crossings by $0-1$ sequences $\mathbf{w}_1$ and the second type by $\mathbf{w}_2$. Then the resolution tube of $\widetilde{L}$ can be encoded by the catenation $\bfv=(\mathbf{w}_1,\mathbf{w}_2)$. The differential splits as $d=d_1+d_2$, where $d_i$ corresponds to the changes of smoothing on type $i$ crossings. Denote the partial resolution of $\widetilde{L}$ on $\bfw_2$ by $\widetilde{L}_{\textbf{w}_2}$, which is also the augmentation of the annular unlink $L_{\bfw_2}$ and hence there is no ambiguity.
	
	The chain complex $\CKhr(\widetilde{L})$ is bigraded by $(|\bfw_1|,|\bfw_2|)$ and the spectral sequence of double complexes applies. The $E_1$ term is given by the cohomology of $(\CKhr(\widetilde{L}),d_1)$, which is a chain complex with chain groups $\Khr(\widetilde{L}_{\bfw_2})$ and differentials given by cobordisms. Since the link diagram is fixed, such cobordisms are corresponding to Morse moves. By Theorem \ref{canonicaliso}, the $E_1$ term is isomorphic to the annular resolution of $L$ and hence the $E_2$ term is isomorphic to $\AKh(L)$. The spectral sequence converges to the cohomology of $(\CKhr(\widetilde{L}),d)$, i.e. $\Khr(\widetilde{L})$. 
	
	A Reidemeister move induces an isomorphism between the converging terms that is compatible with the filtration, and an isomorphism between the $E_2$ terms. The comparison theorem then applies and hence the spectral sequence is independent of the choice of the link diagram. This completes the proof.
\end{proof}

\begin{exmp}
	Consider the annular link $L$ shown in Figure \ref{whitehead}. The augmentation $\widetilde{L}$ is isotopic to the link L5a1 and \[\rankk_{\ZZ}\AKh(L)=8=\rankk_{\ZZ}\Khr(\widetilde{L}).\] Hence the spectral sequence collapses at the $E_2$ term. This illustrates that the spectral sequence might be degenerated for link not isotopic to a braid closure.
\end{exmp}

We can derive a finer rank inequality from Theorem \ref{mainthm}.

\begin{cor}
	Denote $L$ and $\widetilde{L}$ as in Theorem \ref{mainthm}. Let $n_0$ be the number of annular strands and $n_-'$ be the number of left-handed crossings on the augmenting circle. Then we have \[\rankk_{\ZZ}\Khr^{n}(\widetilde{L})\le\sum_{n_a+f_a+n_0-n_-'=n}\rankk_{\ZZ}\AKh^{n_a}(L,f_a).\]
\end{cor}

\begin{proof}
	Denote the gradings of $\AKh(L)$ by $(n_a,q_a,f_a)$. Let $\widetilde{n}_-$ be the number of left-handed crossings of $\widetilde{L}$. Then $\widetilde{n}_-=n_-+n_-'$, and $n_a=|\bfw_2|-n_-$. Let $n_0'$ be the number of nontrivial unknot components of a specific partial resolution. Let $(n_0')_+$ (resp. $(n_0')_-$) be the number of $1$-smoothings (resp. $0$-smoothings). Then by Theorem \ref{canonicaliso}, we have \[f_a=(n_0')_+-n_0'=\frac{(n_0')_+-(n_0')_-}{2}=|\bfw_1|-n_0\] on the $E_1$ term. On the $E_\infty$ term, we have $n=|\bfw_1|+|\bfw_2|-\widetilde{n}_-$. Therefore, from Theorem \ref{mainthm}, we obtain \begin{align*}
		\rankk_{\ZZ}\Khr^n(\widetilde{L})&=\sum_{|\bfw_1|+|\bfw_2|-\widetilde{n}_-=n}\rankk_{\ZZ}E_\infty^{|\bfw_1|,|\bfw_2|}\\
		&\le \sum_{|\bfw_1|+|\bfw_2|-\widetilde{n}_-=n}\rankk_{\ZZ}E_2^{|\bfw_1|,|\bfw_2|}\\
		&=\sum_{n_a+f_a+n_0-n_-'=n}\rankk_{\ZZ}\AKh^{n_a}(L,f_a).
	\end{align*}

\end{proof}

We now prove Theorem \ref{classify} and Corollary \ref{detect}. The following simple observation is useful.

\begin{lem}\label{1strand}
		Let $L$ be an annular link with a link diagram such that there is only one annular strand. View $L$ as a link in $S^3$ and let $p$ be a base point on this annular strand. Then $\AKh(L)$ is supported on $f=\pm 1$, and we have \[\AKh(L,\pm 1)\cong\Khr(L,p).\]
\end{lem}

\begin{proof}
	There is exactly one nontrivial circle in each resolution of $L$, which is the circle containing $p$. Hence the chain complex is supported on $f=\pm 1$. Furthermore, the subcomplexes of $f$-grading $\pm 1$ are isomorphic to $\operatorname{CKhr}(L)$ by replacing the generators $w_\pm$ of the nontrivial circle by $v_-$, respectively.
\end{proof}

\begin{proof}[Proof of Theorem \ref{classify}]
	Let $G$ be a forest such that each connected component contains at most one annular vertex. Then $L_G$ is a disjoint union of links with at most one annular strand. Then Lemma \ref{1strand} applies and we have $\rankk_{\ZZ}\AKh(L_G)=2^n$ by K\"unneth formula.
	
	Conversely, let $L$ be an annular link with $n$ components and \[\rankk_{\ZZ}\AKh(L)=2^n.\] Then Corollary \ref{rankineq} gives $\rankk_{\ZZ}\Khr(\widetilde{L})=2^n$. By \cite[Theorem 1.2]{XZ19b}, $\widetilde{L}$ is a forest of unknot in $S^3$. Therefore, $L$ is a forest of unknot in $A\times I$. Denote their corresponding forests by $\widetilde{G}$ and $G$ respectively. Notice that $\widetilde{G}$ is constructed from $G$ by adding a vertex adjacent to all the annular vertices. Two annular vertices cannot lie in the same connected component of $G$ since otherwise a cycle would occur in $\widetilde{G}$, which is absurd since $\widetilde{G}$ is a forest.
\end{proof}

\begin{proof}[Proof of Corollary \ref{detect}]
	By Theorem \ref{classify}, $L$ is a forest of unknots in $A\times I$. Denote the corresponding forest by $G$. If $G$ has an edge, then $\AKh(L)$ would not be supported on $t=0$ as $\AKh(U)$ does (see the discussion in Section 3.1), which is a contradiction. Hence every vertex is an independent connected component of $G$, i.e. $L$ is an annular unlink. The number of nontrivial unknot components in $L$ can be read from the Poincar\'e polynomial of $L$. Therefore $L$ is isotopic to $U$. 
\end{proof}

\bibliographystyle{hplain}
\bibliography{ref}

\end{document}